\documentclass[11pt]{amsart}
\usepackage{amsfonts}
\usepackage{amssymb}
\usepackage{amsmath,amsthm}
\usepackage{amscd}
\usepackage{graphics}
\usepackage{amsmath,amssymb}
\usepackage{pdfsync}
\usepackage{mathrsfs}
\usepackage{enumitem}
\usepackage{stmaryrd}
\usepackage[colorlinks,linkcolor=blue,citecolor=blue,urlcolor=blue]{hyperref}

\newcommand{\beq}{\begin{equation}}
\newcommand{\eeq}{\end{equation}}
\newcommand{\beqa}{\begin{eqnarray}}
\newcommand{\eeqa}{\end{eqnarray}}
\newcommand{\nn}{\nonumber}

\newcommand{\bx}{\mathbf{x}}
\newcommand{\by}{\mathbf{y}}
\newcommand{\bz}{\mathbf{z}}
\newcommand{\mss}{\kern 1pt}

\newcommand{\NN}{{\mathbb N}}

\newtheorem{definition}{Definition}
\newtheorem{proposition}{Proposition}
\newtheorem{theorem}{Theorem}
\newtheorem{corollary}{Corollary}
\newtheorem{lemma}{Lemma}
\theoremstyle{definition}

\newtheorem{remark}{\textbf{Remark}}

\newtheorem{problem}{\textbf{Problem}}
\newtheorem{example}{\textbf{Example}}

\begin{document}

\author{Jose Carrasco}
\address{Instituto de Ciencias Matem\'aticas, C/ Nicol\'as Cabrera, No 13--15, 28049 Madrid, Spain\\ and Departamento de F\'{\i}sica Te\'{o}rica, Facultad de Ciencias F\'{\i}sicas, Universidad Complutense de Madrid, 28040 -- Madrid, Spain}
\email{joseacar@ucm.es}

\author{Piergiulio Tempesta}
\address{Instituto de Ciencias Matem\'aticas, C/ Nicol\'as Cabrera, No 13--15, 28049 Madrid, Spain\\ and Departamento de F\'{\i}sica Te\'{o}rica, Facultad de Ciencias F\'{\i}sicas, Universidad
Complutense de Madrid, 28040 -- Madrid, Spain }
\email{piergiulio.tempesta@icmat.es, ptempest@ucm.es}

\subjclass[2010]{MSC: 14L05, 13J05} \title{Formal rings}
\date{February 10, 2019}

\begin{abstract}
  A notion of one-dimensional formal ring is presented. It consists of
  a triple $(A,\Phi,\Psi)$ where $A$ is a unital ring and
  $\Phi$ and $\Psi$ are two formal
  power series in $2$ variables ${\Phi(x,y),\Psi(x,y)\in A\llbracket x,y\rrbracket}$, the first one defining a
  one-dimensional formal group law over $A$ and the second one
  providing a second composition law satisfying axiomatic properties
  of compatibility with the first one.

  For a characteristic-zero ring $A$, a large class of one-dimensional
  formal rings can be obtained by constructing a new composition law,
  defined in terms of the group logarithm associated with a given
  formal group law $\Phi(x,y)$, and associative and distributive with
  respect to it.

  A natural $n$-dimensional generalization of the previous
  construction is also proposed; curves on $n$-dimensional formal
  rings are introduced.  The higher-dimensional theory allows us to define a generalization of the ring $W(A)$ of  Witt vectors over a ring $A$,
  which is recovered by means of a specific choice of the associated group logarithm. The composition laws of our generalized Witt ring are defined in terms of an underlying formal ring structure.

  Examples of formal rings related to
  Hirzebruch's theory of genera are explicitly computed. Finally, we also
  propose the examples of Euler's and Abel's formal rings.
\end{abstract}

\maketitle

\tableofcontents

\section{Introduction}\label{sec.int}
Since the seminal paper by Bochner~\cite{Bochner}, formal group theory
has been largely investigated due to its prominent role in several
branches of mathematics as algebraic topology~\cite{BN}-\cite{Lazard},
cobordism theory~\cite{Nov},~\cite{Quillen}, field theory~\cite{LT},
analytic number theory~\cite{Honda}, \cite{Serre1966},
\cite{PT2010}-\cite{PT2015}, etc. (see \cite{Haze} and
\cite{Serre1992} for a thorough exposition). Recently, some
applications in statistics and information geometry have also been
found~\cite{P2016}.

The aim of this article is to present the construction of certain
algebraic structures related to formal group laws, that we shall call
\textit{formal rings}. In order to illustrate this notion, we consider
first the one-dimensional case.  Given a unital commutative ring $A$,
a one-dimensional formal ring is a triple $(A,\Phi,\Psi)$ where
$\Phi(x,y)\in A\llbracket x,y\rrbracket$ is a one-dimensional formal
group law in two variables over $A$ and
$\Psi(x,y)\in A\llbracket x,y\rrbracket $ is another formal power
series in the same variables, representing a second composition
law. This law is required to be associative, commutative and
distributive with respect to the first one.

In the case of a characteristic-zero ring $A$, a straightforward
theorem of existence of formal ring structures holds: given a
one-dimensional commutative formal group law $\Phi(x,y)$ over $A$,
there exists a second composition law expressed in terms of another
two-variable formal power series
$\Psi(x,y)\in A\llbracket x,y\rrbracket$, associative and distributive
with respect to the first one, whose coefficients can be explicitly
computed in terms of the group logarithm associated with the first
law.

By analogy with the theory of formal groups, a notion of homomorphism
between formal rings is introduced. This allows us to define, for any
ring $A$ the category $\mathcal{C}(A)$ of formal rings over $A$. We
shall observe that a ring homomorphism $r:A_1\to A_2$ naturally
induces a functor from $\mathcal{C}(A_1)$ to $\mathcal{C}(A_2)$. We
shall also propose the construction of a {``canonical"} isomorphism
between two given formal rings over a characteristic zero ring $A$,
which is inspired by the theory of $A$-typical formal $A$-modules.

The construction of one-dimensional formal rings proposed is also
discussed in relation with the theory of Hirzebruch's multiplicative
genera. The formal rings related with the $c$-genus, the Todd genus
and the $L$-genus are explicitly determined.

Besides, we propose the notion of $n$-dimensional formal ring as a
natural generalization of the previous construction. This algebraic
structure can be realized in several ways. In particular, given an
$n$-dimensional formal group law and the group logarithm associated
with it, we prove that there exists a second family of formal power
series in $2n$ variables, defined by means of the same group
logarithm, which is associative and distributive with respect to the
given formal group. In other words, over a unital commutative ring $A$ of characteristic
zero, with any formal group law one can associate a formal ring
structure, which is defined a priori in the same ring where the group
logarithm is defined.

We also propose a concept of curves defined on formal rings. The
composition laws of each formal ring endow the set of its curves with
a standard ring structure. Also, a filtration of subrings can be
defined, which makes it possible to introduce a topology in the set of
curves.

An important result of the theory of formal rings is the possibility of defining a generalization of the standard ring of Witt vectors over a unital ring $A$. Indeed,the ring $W(A)$ and its truncated version $W_n(A)$ can be interpreted as a particular instance of a more general formal ring structure, that we shall call a \textit{generalized Witt ring} $GW(A)$. This structure is defined by assigning a suitable ``triangular'' form to the components of the underlying $n$-dimensional group logarithm, and can be explicitly computed by means of an iterative process.
\par
The standard definition of the ring $W_n(A)$ is  recovered when the components of group logarithm are identified with the usual \textit{ghosts} associated with the Witt vectors. The same result holds for the case of the universal Witt vectors.


Many questions deserve to be further investigated. A fundamental,
general one is to establish the notion of formal ring over a field of
characteristic $p>0$.  A straightforward question is to ascertain
whether there exists a construction of formal rings over a ring $A$ of
integers of a discretely valued field $K$, analogous to the elegant
constructions of formal group laws proposed by Honda \cite{Honda1970}
and Lubin-Tate \cite{LT}.  Possible applications of formal rings in
statistical mechanics, along the lines of \cite{P2016} are certainly
of interest. An infinite-dimensional generalization of the notion of
formal ring would also be of potential relevance in the applications.

The article is organized as follows. In Section~\ref{sec.1dim}, we will recall some
basic facts about formal group theory and introduce the notion of
formal ring in the one-dimensional case. In Section~\ref{sec.frings}, we shall prove
our main existence theorem of the formal ring structure. In Section~\ref{sec.hom},
we introduce the notion of homomorphism of formal rings and the
category of formal rings.  In Section~\ref{sec.ndim}, we generalize our
construction to the $n$-dimensional case; curves over $n$-dimensional
formal rings are also defined.  In Section ~\ref{sec.Witt}, we prove that the ring of Witt vectors can be obtained from a more general ring structure,
whose construction is explicitly proposed. The examples of the Euler and the Abel
formal rings, as well as those related with Hirzebruch's theory of
genera, are discussed explicitly in Section~\ref{sec.hir}. A two dimensional
bi-parametric formal ring generalizing the multiplicative group law is
explicitly computed in the final Section~\ref{sec.2dim}.

\section{One-dimensional formal group laws and rings}\label{sec.1dim}
In order to fix the notation, we shall first recall some basic facts
about formal group theory.

\subsection{Formal group laws} \cite{Bochner}, \cite{Haze},
\cite{Serre1992}.  All rings considered will be associative
commutative unital rings.  In particular, given a ring $(A,+,\cdot)$,
we shall denote by $0$ the neutral element of the addition operation
$+:A\times A\to A$ and by $1$ the neutral element of the
multiplication operation $\cdot:A\times A\to A$.

A commutative one--dimensional formal group law over a ring $A$ is a
formal power series~$\Phi(x,y)\in A\llbracket x,y\rrbracket$ such that
\begin{enumerate}[label=\roman*)]
\item $\Phi(x,0)=\Phi(0,x)=x\,,$
\item $\Phi\left(\Phi(x,y),z\right)=\Phi\left(x,\Phi(y,z)\right)\,.$
\end{enumerate}
When $\Phi(x,y)=\Phi(y,x)$, the formal group law is said to be
commutative.

Let $B=\mathbb{Z}[b_{1},b_{2},\ldots]$ be the ring of integral
polynomials in infinitely many variables.  We shall consider the
formal series $F_{U}(s)=\sum_{i=0}^\infty b_i\frac{s^{i+1}}{i+1}$ with
$b_0=1$ (the \textit{group exponential}). Let
\begin{equation*}
G_{U}(t)=\sum_{k=0}^\infty a_k\frac{t^{k+1}}{k+1}
\end{equation*}
be its compositional inverse, i.e., $F_{U}\left(G_{U}(t)\right)=t$
(the associated \textit{group logarithm}). From this property we
deduce $a_0=1, a_1=-b_1, a_2=\frac32 b_1^2-b_2,\ldots$ for the
coefficients of $G_{U}(t)$.  The Lazard universal formal group law~\cite{Haze} is defined by the formal power series
\begin{equation*}
  \Phi_U(s_1,s_2)=G_{U}^{-1}\left(G_{U}(s_1)+G_{U}(s_2)\right)\,.
\end{equation*}
The coefficients of the power series
$G_{U}^{-1}\left(G_{U}(s_1)+G_{U}(s_2)\right)$ lie in the ring
$B\otimes\mathbb Q$ and generate over $\mathbb Z$ a subring
$L\subset B\otimes\mathbb Q$, called the Lazard ring. For any
commutative one-dimensional formal group law over any ring $A$, there
exists a unique homomorphism $L\to A$ under which the Lazard group law
is mapped into the given group law (\textit{universal property} of the
Lazard group).

Let $A$ be a ring of characteristic zero. Then, for any commutative
one-dimensional formal group law $\Phi(x,y)$ over $A$, there always
exists a group logarithm, namely a series
$f(x)\in A\llbracket x\rrbracket\otimes\mathbb{Q}$ such that
\begin{equation*}
  f(x)= x+O(x^2)\quad\text{and}\quad\Phi(x,y)=f^{-1}\left(f(x)+f(y)\right)\in A\llbracket x,y\rrbracket\otimes\mathbb Q\,.
\end{equation*}
We shall present here our main definition, which naturally generalizes
the notion of formal group law.

\subsection{Formal rings}
We propose here our main definition, that we shall first formulate in the one-dimensional case.
\begin{definition}\label{def1}
  Let $\left(A,+,\cdot\right)$ be a unital ring. A formal ring is a
  triple ${\mathcal{R}:=(A,\Phi,\Psi)}$ where
  $\Phi\left(x,y\right),\Psi(x,y)\in A\llbracket x,y\rrbracket$ are
  formal power series such that
  \begin{enumerate}[label=\roman*)]
  \item $\Phi(x,y)$ is a commutative formal group law
  \item $\Psi(x,y)$ satisfies the following relations:
    \begin{align*}
      \Psi(\Psi(x,y),z)&=\Psi(x,\Psi(y,z))\,,\\
      \Psi(x,\Phi(y,z))&=\Phi(\Psi(x,y),\Psi(x,z))\,,\\
      \Psi(\Phi(x,y),z)&=\Phi(\Psi(x,z),\Psi(y,z))\,.
    \end{align*}
  \end{enumerate}
  In particular, whenever $\Psi(x,y)=\Psi(y,x)$, the formal ring will
  be said to be commutative.
\end{definition}
Hereafter, we shall say that $A$ is the \textit{base ring}
of the formal ring $\mathcal{R}=(A,\Phi,\Psi)$.

\section{Formal rings from the Lazard universal formal group}\label{sec.frings}
There is a general construction of formal rings, obtained by extending
the Lazard universal formal group by means of a binary operation
defined in terms of the group logarithm and exponential.

\subsection{Existence of formal rings}
\begin{theorem}\label{teo}
  Let $B=\mathbb{Z}[a_{1},a_{2},\ldots]$. The group logarithm
  given by the formal power series $G_{U}(x)=\sum_{k=0}^\infty a_k\frac{x^{k+1}}{k+1}$ induces two
  binary operations
  \begin{align}
    \Phi_U(x,y)&:=G_{U}^{-1}\left(G_{U}(x)+G_{U}(y)\right)\qquad
                 \text{(addition)}\label{Phi}\\
    \Psi_U(x,y)&:=G_{U}^{-1}\left(G_{U}(x)\cdot G_{U}(y)\right)\qquad
                 \text{(multiplication)}\label{Psi}
  \end{align}
  which define the structure of a formal ring $\mathcal{R}=(A,\Phi_U,\Psi_U)$ over the base ring $A=B\otimes\mathbb{Q}\,$.
\end{theorem}

\begin{proof}
  The operations
  $\Phi_U(x,y), \Psi_U(x,y)\in A\llbracket x,y\rrbracket$ are
  obviously well defined in terms of addition, multiplication and
  composition of formal power series.  From
  Eqs.~\eqref{Phi}-\eqref{Psi} we deduce the distributivity property
  \[
    \Psi_U\left(x,\Phi_U(y,z)\right)=\Phi_U\left(\Psi_U(x,y),\Psi_U(x,z)\right).
  \]
  Also, introducing
  \[\chi(x):=G^{-1}\left(-G(x)\right)\in A\llbracket x\rrbracket\,,
  \]
  we have that
  $\Phi_U\left(x,\chi(x)\right)=\Phi_U\left(\chi(x),x\right)=0$, which
  shows that $0$ is the neutral element of the group law $\Phi_U$.
\end{proof}

We remark that the formal ring $(A,\Phi_U,\Psi_U)$ a priori is not
unital. However, if $G^{-1}(1)\in A$, the relations
$\Psi_U(x,G^{-1}(1))=\Psi_U(G^{-1}(1),x)=x$ make
sense, showing that $G^{-1}(1)$ behaves as the ``unit'' of operation $\Psi_U$, the multiplication.

A natural problem is to establish if all one-dimensional formal rings
over a commutative ring $A$ of characteristic zero can be obtained
according to the previous construction. In other words, for a
commutative formal group law $\Phi(x,y)$ with associated group
logarithm $G(x)$, we wish to determine the most general form of the
composition law $\Psi(x,y)$ which fulfills Definition~\ref{def1}. To
this aim, by imposing the associativity and distributivity laws we can
arrive to the following equivalent

\begin{problem}
  \textit{Given a group logarithm $G(x)$, to find the most general
    solution $F(x,y)$ to the functional equations}
  \begin{align*}
    F(x,y)+F(x,z)&=F(x,G^{-1}(G(y)+G(z)),\\
    F(x,G^{-1}(F(y,z)))&=F(G^{-1}(F(x,y)),z).
  \end{align*}
\end{problem}

Notice that any solution $F(x,y)$ of the previous problem induces a
second composition law $\Psi(x,y)$ which is compatible with the given
one, namely the formal group law $\Phi(x,y)=G^{-1}(G(x)+G(y))$, by the relation
$F(x,y)= G(\Psi(x,y))$. So that $(A, \Phi, \Psi)$ is a formal ring. A
partial answer to Problem 1 will be discussed in {Section~\ref{sec.hom}}.

Another important open problem is to ascertain whether, at least for a
specific class of formal rings, one could define a base ring playing a
role analogous to that of the Lazard ring $L$ in formal group theory.
As is well known, by means of the functional-equation lemma~\cite{Haze}, under certain conditions one can generate formal group
laws over a suitably defined ring $A$, even if the corresponding group
logarithm is defined over a larger ring $B \supset A$.  However, at
least in the case of one-dimensional formal rings, we have proved that
the base ring for the formal ring structure obtained in Theorem~\ref{teo} exactly coincides with the ring over which the corresponding
group logarithm is defined.

\subsection{Commutative fields}
Let $\left(K,+,\cdot\right)$ be a field of characteristic zero and let
$\{0,1\}$ denote the neutral elements of the addition and the
multiplication, respectively. Given a bijective function over $K$,
using it we can construct a new commutative field as follows.

\begin{proposition}\label{cor}
  Every bijective function $h: K\to K$ with $h(0)=0$ determines two
  binary operations $\Phi$ and $\Psi$ in $K$, namely the addition~$\Phi(x,y)=h^{-1}\left(h(x)+h(y)\right)$ and the multiplication
  $\Psi(x,y)= h^{-1}\left(h(x)\cdot h(y)\right)$, in such a way that
  $(K,\Phi,\Psi)$ is a commutative field.  Moreover, $0$ is the
  neutral element of addition $\Phi$ and $h^{-1}(1)$ is the neutral
  element of multiplication $\Psi$.
\end{proposition}

\begin{proof}
  We only need to show that for any $x\in K$ with $x\neq 0$ there
  exists an inverse $\psi(x)\in K$ such that
  $\Psi\left(x,\psi(x)\right)=\Psi\left(\psi(x),x\right)=h^{-1}(1)$. Since
  $h(x)$ is a bijection of $K$, $h(x)\neq 0$ for $x\neq 0$ and
  $1/h(x)$ is well defined.  Thus, we can introduce
  $\psi(x):=h^{-1}(1/h(x))$.  It follows from the definition of $\Psi$
  that
  $\Psi\left(x,\psi(x)\right)=\Psi\left(\psi(x),x\right)=h^{-1}(1)$.
\end{proof}

\section{Homomorphisms and categories of formal rings}\label{sec.hom}
\subsection{Preliminaries}
By analogy with the categorical approach to formal groups~\cite{Haze},
we can introduce a category of formal rings. Preliminarily, we shall
propose a natural definition of \textit{homomorphism of formal rings}.
This definition extends the well-known notion of a homomorphism
$\chi: \Phi_1 \to \Phi_2$ of formal groups over a ring $A$ (see the
discussion in~\cite{Haze}), where
$\chi(x) \in A\llbracket x\rrbracket$ is a formal series without
constant term, satisfying the relation
$\chi\big(\Phi_1(x,y)\big)= \Phi_2 \big(\chi(x), \chi(y)\big)$.

\begin{definition} Given two formal rings
  $\mathcal{R}_1=(A, \Phi_1, \Psi_1)$ and
  $\mathcal{R}_2=(A, \Phi_2, \Psi_2)$ over a ring $A$, a formal
  series 
  $\phi(x)\in A\llbracket x\rrbracket$ with $\phi(0)=0$, such that
  \begin{align*}
    \phi(\Phi_1(x,y))&=\Phi_2(\phi(x),\phi(y)) \\
    \phi(\Psi_1(x,y))&=\Psi_2(\phi(x),\phi(y))
  \end{align*}
  will be said to be a homomorphism of the formal rings
  $\mathcal{R}_1$, $\mathcal{R}_2$.
\end{definition}

Let $\phi(x)= b_1 x+ b_2 x^2+ \ldots$ be a homomorphism of rings. If
$b_1=1$, then the homomorphism $\phi(x)$ is an isomorphism, that will
be called a \textit{strict} isomorphism.  For any couple of
homomorphisms $\phi_1: \mathcal{R}\rightarrow\mathcal{R}'$ and
$\phi_2:\mathcal{R}'\rightarrow\mathcal{R}''$, a composition of
homomorphisms can be clearly defined in terms of the formal series
$\phi_2 (\phi_1 (x))$. We shall denote by
$\mathrm{Hom}_{A}(\mathcal{R},\mathcal{R'})$ the set of all
homomorphisms from $\mathcal{R}$ into $\mathcal{R}'$ over the ring
$A$. Also, we introduce the category $\mathcal{C}_{A}(\mathcal{R})$
defined by all formal rings over $A$ with their associated
homomorphisms.

In the previous construction, all formal rings were defined over the
same base ring $A$. Another class of applications among formal rings
arises when we consider the action of homomorphisms
$r:A_1\rightarrow A_2$ between different base rings. Let $\mathcal{R}_1=(A_1, \Phi, \Psi)$ be a formal ring over $A_1$,
where
\[
\Phi(x,y)= x+y+\sum_{ij}\alpha_{ij}x^iy^j, \qquad
\Psi(x,y)= x\cdot y+\sum_{ij}\beta_{ij}x^iy^j\,,
\]
and let $r:A_1\rightarrow A_2$ be a homomorphism of rings.  It naturally
induces an application
$ r_{*}: A_1\llbracket x,y\rrbracket\rightarrow A_2\llbracket
x,y\rrbracket $ defined as
\begin{equation}\label{rhoinduced}
  r_{*}[\Phi](x,y):= x+y+\sum r(\alpha_{ij}) x^{i} y^{j}, \quad r_{*}[\Psi](x,y):= x\cdot y+\sum r(\beta_{ij}) x^{i} y^{j} \ .
\end{equation}
The following propositions are a direct consequence of the fact that a
ring homomorphism preserves the algebraic relations among the
coefficients of the two compatible formal series defining a formal
ring structure $(A, \Phi, \Psi)$.

\begin{proposition}
  Let $\mathcal{R}_1:=(A_1,\Phi,\Psi)$ be a formal ring over
  $A_1$. Also, let $r: A_1\rightarrow A_2$ be a homomorphism of
  rings. Then $\mathcal{R}_2:=(A_2, r_{*}[\Phi], r_{*}[\Psi])$ is a
  formal ring over the base $A_2$.
\end{proposition}

\begin{proposition}
  Let $\phi(x)=\sum\beta_ix^i$ be a homomorphism
  $\phi:\mathcal{R}\rightarrow\mathcal{R}'$ of formal rings over
  $A_1$. Then, given a homomorphism $r: A_1\rightarrow A_2$, the
  series $r_{*}[\phi](x)=\sum_ir(\beta_i)x^i$ defines a homomorphism
  $r_{*}[\phi]$ of formal rings over $A_2$.
\end{proposition}

We can conclude that any homomorphism $r: A_1\to A_2$ naturally
induces a functor from $\mathcal{C}_{A_1}(\mathcal{R}_1)$ into
$\mathcal{C}_{A_2}(\mathcal{R}_2)$.

\subsection{A canonical isomorphism of formal rings}
We shall prove the existence of a standard isomorphism between two
given one-dimensional formal rings. Precisely, let
$\mathcal{R}_1=(A, \Phi_1, \Psi_1)$ and
$\mathcal{R}_2=(A, \Phi_2, \Psi_2)$ be two formal rings over a ring
$A$ of characteristic zero. Let $G_1(t)$ and $G_2(t)$ be the group
logarithms associated with $\Phi_1(x,y)$ and $\Phi_2(x,y)$,
respectively. Inspired by the theory of $A$-typical formal
$A$-modules, we introduce the application
\[
\sigma_a(x):=G_2^{-1}(a~G_{1}(x)).
\]
One can easily prove the following result.

\begin{lemma} \label{lemma4} The following relations hold:
  \begin{align*}
    \sigma_{a}(\Phi_1(x,y)) &=  \Phi_2(\sigma_{a}(x), \sigma_{a}(y))\,,\\
    \sigma_{ab}(\Psi_1(x,y))&= \Psi_2(\sigma_{a}(x), \sigma_{b}(y))\,.
  \end{align*}
\end{lemma}

\begin{proof}
  Consider the following chain of equalities in order to prove
  the first relation:
  \[
    \begin{aligned}
      (G_2\circ\sigma_{a}\circ\Phi_1)(x,y)&=a\,(G_1\circ\Phi_1)(x,y)=a\,G_1(x)+a\,G_1(y)\\
      &=(G_2\circ\sigma_a)(x)+(G_2\circ\sigma_a)(y)\\
      &=(G_2\circ\Phi_2)\left(\sigma_a(x),\sigma_a(y)\right).
    \end{aligned}
  \]
  Also, for the second one:
  \[
    \begin{aligned}
      (G_2\circ\sigma_{ab}\circ\Psi_1)(x,y))&=ab\,(G_1\circ\Psi_1)(x,y)=a\,G_1(x)\,b\,G_1(y)\\
      &=(G_2\circ\sigma_a)(x) \cdot\,(G_2 \circ
      \sigma_b)(y)\\
      &=(G_2\circ\Psi_2)\left(\sigma_a(x),\sigma_b(y)\right).
    \end{aligned}
  \]
  Since $G_2$ is invertible, the result follows.
\end{proof}

Let $\sigma_{1}(x)\equiv \sigma(x)$. Thus, as an immediate consequence
of Lemma \ref{lemma4}, we can define a strict isomorphism.

\begin{corollary}
  let $\mathcal{R}_1=(A, \Phi_1, \Psi_1)$ and
  $\mathcal{R}_2=(A, \Phi_2, \Psi_2)$ two formal rings over $A$. Let
  $G_1(t)$ and $G_2(t)$ be the group logarithms associated with
  $\Phi_1(x,y)$ and $\Phi_2(x,y)$, respectively.  Then the function
  $\sigma(x)=G_2^{-1}(G_{1}(x))$ is a isomorphism between
  $\mathcal{R}_1$ and $\mathcal{R}_2$.
\end{corollary}

We shall call \textit{canonical} this isomorphism. An interesting
particular case is obtained when $G_1=G_2=G$. In this case,
$\sigma_a(x)$ reduces to the well known function
$\rho_{a}(x)=G^{-1}(aG(x))$, which plays a crucial role un the
standard theory of formal groups, since it endows a formal group
$\Phi(x,y)$ with the structure of $A$-module.  In order to clarify the
role of $\rho_{a}(x)$ in the theory of formal rings we observe that,
for $\mathcal{R}=(A, \Phi, \Psi)$ formal ring over $A$ with group
logarithm $G(t)$, one can easily prove the following result.

\begin{lemma}
  The following relations hold:
  \begin{align}
    \Phi(\rho_{a}(x), \rho_{a}(y))&= \rho_{a}(\Phi(x,y))\,,\label{rel1}\\
    \Psi(\rho_{a}(x), \rho_{b}(y))&= \rho_{ab}(\Psi(x,y))\,,\label{rel2}\\
    \Psi(\rho_{a}(x), \rho_{b}(x))&= \Psi(\rho_{ab}(x),x))\,.\label{rel3}
  \end{align}
\end{lemma}

\begin{proof}
  Relation \eqref{rel1} is well known \cite{Haze}. Concerning
  Eq.~\eqref{rel2}, we observe that
  \[
    \begin{aligned}
      (G\circ\Psi)\left(\rho_{a}(x),\rho_{b}(y)\right)&=G(\rho_{a}(x))\,G(\rho_{b}(y))=a\,G(x)\,b\,G(y)\\
      &=ab\,G(x)G(y)=(G\circ\rho_{ab}\circ\Psi)(x,y)
    \end{aligned}
  \]
  and apply $G^{-1}$ to both sides. Relation~\eqref{rel3} can be
  proved similarly.
\end{proof}

Now, let $\mathcal{R}=(A, \Phi, \Psi)$ be a formal ring over a ring of
characteristic zero $A$, with associated group logarithm $G(x)$. Let
us introduce the formal power series
\[
\Psi_{a}(x,y):=
\rho_{a}(\Psi(x,y))=\Psi(\rho_{a}(x),y)= G^{-1}(G(x)\,a\,G(y)) \ .
\]
One can prove directly the following

\begin{proposition}
  For every $a \in A$, the triple
  $\mathcal{R}_{a}=(A, \Phi, \Psi_{a})$ is a formal ring homomorphic
  to the ring $\mathcal{R}=(A, \Phi, \Psi)$.
\end{proposition}

\begin{definition} \label{a-formal ring} The formal ring
  $\mathcal{R}_{a}=(A, \Phi, \Psi_{a})$ will be said to be the
  $a$-ring associated with $\mathcal{R}$.
\end{definition}

\begin{remark}
  From the previous discussion we can infer that Problem 1 above does
  not admit a unique solution. Indeed, if $\Psi(x,y)$ is compatible
  with a given formal group law $\Phi(x,y)$, i.e. they generate a
  formal ring structure, then for any $a\in A$ we have that
  $\Psi_{a}(x,y)$ is also compatible; consequently, the $a$-ring
  associated provides another solution to the problem.
\end{remark}

\section{N-dimensional formal rings}\label{sec.ndim}
In this section, we shall propose an $n$-dimensional generalization of
the previous construction. We recall that an $n$--dimensional formal group law $\Phi=(\Phi_1,\ldots,\Phi_n)$ over a ring $R$ is
an $n$-tuple of formal power series
$\Phi_i(\bx,\by)$ in the $2n$ indeterminates
$\bx=(x_1,\ldots,x_n)$, $\by=(y_1,\ldots,y_n)$ such that
\begin{enumerate}[label=\roman*)]
\item
  $\Phi_i(\bx,\by)=x_i+y_i \hspace{3mm}\text{mod degree 2
    terms},$
\item
  $\Phi\left(\Phi(\bx,\by),\bz\right)=\Phi\left(\bx,\Phi(\by,\bz)\right).$
\end{enumerate}
When $\Phi_i(\bx,\by)=\Phi_i(\by,\bx)$ for $i=1,\ldots,n$ the
formal group law is said to be commutative.

By analogy with the one-dimensional construction previously developed,
we shall define below a natural notion of an $n$-dimensional formal
ring structure, defined starting from an underlying $n$-dimensional
formal group law.

\subsection{Definition and main properties}

\begin{definition}\label{ndimfr}
  Let $\left(A,+,\cdot\right)$ be a commutative unital ring. An
  $n$-dimensional formal ring over $A$ is a triple
  $\mathcal{R}:=(A,\Phi,\Psi)$ where $\Phi=(\Phi_1,\ldots,\Phi_n)$ and $\Psi=(\Psi_1,\ldots,\Psi_n)$ are two sets of $n$ formal series $\Phi_i(\bx,\by)$ and $\Psi_i(\bx,\by)$ each
  of them in $2n$ variables $(x_1,\ldots,x_n,y_1,\ldots,y_n)$ such that
  \begin{enumerate}[label=\roman*)]
  \item $\Phi(\bx,\by)$ defines a commutative $n$-dimensional
    formal group law.
  \item
    $\Psi\left(\Psi(\bx,\by),\bz\right)=\Psi\left(\bx,\Psi(\by,\bz)\right).$
  \item The following distributivity property holds:
    \begin{align*} \Psi(\bx,\Phi(\by,\bz))&=\Phi(\Psi(\bx,\by),\Psi(\bx,\bz)), \\
      \Psi(\Phi(\bx,\by),\bz)&=\Phi(\Psi(\bx,\bz),\Psi(\by,\bz)).
    \end{align*}
  \end{enumerate}
  In particular, whenever $\Psi(\bx,\by)=\Psi(\by,\bx)$,
  the formal ring will be said to be commutative.
\end{definition}

\begin{example}
  A simple realization of an $n$-dimensional formal ring is obtained
  by choosing $\Phi=(\Phi_1,\ldots,\Phi_n)$ and
  $\Psi=(\Psi_1,\ldots,\Psi_n)$ as follows. Consider $n$ pairs
  $\tilde\Phi_i\,,\tilde\Psi_i\in A\llbracket x,y\rrbracket$ of formal
  power series defining $n$ one-dimensional formal rings. With a
  slight abuse of notation let
  $\Phi_i(\bx,\by)\equiv\tilde\Phi_i(x_i,y_i)$ and
  $\Psi_i(\bx,\by)\equiv\tilde\Psi_i(x_i,y_i)$ for
  $i=1,\ldots,n$. It follows that the pair $\Phi,\Psi$ is an
  $n$-dimensional formal ring. A trivial realization obviously arises
  from the additive $n$-dimensional formal group law
  $\Phi_i(\bx,\by)= x_i+y_i$ by choosing
  $\Psi_i(\bx,\by)= x_i\cdot y_i$.
\end{example}

A natural question is whether there exists a general approach allowing
us to construct a suitable product compatible with a given
$n$-dimensional formal group law $\Phi(x,y)$.  Under the hypothesis
that the base ring $A$ is of characteristic zero, which ensures that
the group law admits a group logarithm, the following procedure
holds. Given a family $G=(G_{1},\ldots, G_n)$ of formal power series in $n$ variables $\bx=(x_1,\ldots,x_n)$ with
\begin{equation}\label{Gindim}
G_{i}(\bx)=x_i+\sum_{s_1+\cdots s_n>1}\gamma_{(s_1,\ldots,s_n)}\,x_{1}^{s_1}\,\cdots\,x_{n}^{s_n} \ ,
\end{equation}
we first define an $n$-dimensional formal group law by $\Phi(\bx,\by):=G^{-1}(G(\bx)+G(\by))$ where the inverse $n$-dimensional formal power series is obtained by means of the standard Lagrange
inversion procedure.  Then we introduce the set of $n$ formal power
series $\Psi(\bx,\by)=(\Psi_{1}(\bx,\by),\ldots,\Psi_{n}(\bx,\by))$
defined by
\begin{equation}\label{PsindimG}
\Psi(\bx,\by):=
G^{-1}\big(G_1(\bx)\mss G_1(\by),\ldots,
G_n(\bx)\mss G_n(\by)\big),
\end{equation}
so that we can state the following result.

\begin{theorem} \label{theomain} Let $A$ be a characteristic zero
  ring and let $G=(G_1,\ldots,G_n)$ be a set of formal power series where each $G_i(\bx)$ can be written as in Eq.~\eqref{Gindim}. Let~$\Phi(\bx,\by)=G^{-1}{(G(\bx)+G(\by))}$, and $\Psi(\bx,\by)$ be the set of $n$ formal series~\eqref{PsindimG}. Then the triple $\mathcal{R}^{(G)}_{n}:=(A, \Phi, \Psi)$ is an $n$-dimensional
  commutative formal ring.
\end{theorem}

\begin{proof}
  Commutativity of $\Phi(\bx,\by)$ and $\Psi(\bx,\by)$ is obvious. Let us prove associativity. The
  following identities hold:
  \begin{equation}\label{ident}
  G_{i}
  \big(G^{-1}(\bx)\big)=(G^{-1})_i\big(G(\bx)\big)=x_i\,.
  \end{equation}
  Thus, from~\eqref{PsindimG} we obtain
  \[
    \begin{aligned}
      (G_i\circ\Psi)(\bx,\Psi(\by,\bz))&=G_i(\bx)\,(G_i\circ\Psi)(\by,\bz)=G_i(\bx)\,G_i(\by)\,G_i(\bz)\\
      &=(G_i\circ\Psi)(\bx,\by)\,G_i(\bz)=(G_i\circ\Psi)(\Psi(\bx,\by),\bz)\,.
    \end{aligned}
  \]
  Concerning distributivity,
  \[
    \begin{aligned}
      (G_i\circ\Psi)(\bx,\Phi(\by,\bz))&=G_i(\bx)\,(G_i\circ\Phi)(\by,\bz)=G_i(\bx)\,G_i(\by)+G_i(\bx)\,G_i(\bz)\\
      &=(G_i\circ\Psi)(\bx,\by)+(G_i\circ\Psi)(\bx,\bz)\\
      &=(G_i\circ\Phi)\left(\Psi(\bx,\by),\Psi(\bx,\bz)\right).
    \end{aligned}
  \]
  The proof is completed by noting that $G_i\circ\Psi(\bz)=G_i\circ\Psi_i(\bz')$ implies $G\circ\Psi(\bz)=G\circ\Psi(\bz')$ and then $\Psi(\bz)=\Psi(\bz')$. In the same vein, $G_i\circ\Psi(\bz)=G_i\circ\Phi_i(\bz')$ implies $\Psi(\bz)=\Phi(\bz')$.
\end{proof}

\subsection{Curves in formal rings}
In the standard theory of curves defined in formal groups, one can
regard the power series $\Phi(x,y)$ as a recipe for manifacturing
ordinary groups. Indeed, computing a formal group law over two power
series $\gamma_1(t)$, $\gamma_2(t)$ can be interpreted as a ``sum'' of
curves. By analogy, one can generalize this idea to formal rings.

\begin{definition}\label{defcurves}
  Let $\mathcal{R}=(A, \Phi,\Psi)$ be an $n$-dimensional formal ring
  over a unital ring $A$. A curve in $\mathcal{R}$ is an $n$-tuple of
  formal power series $\gamma(t)\in A\llbracket t\rrbracket$ in one
  variable $t$ such that $\gamma(0)=0$.
\end{definition}

Thus, we can use the power series $\Phi$ and $\Psi$ to define an
addition of curves
\[
\gamma_{1}(t) +_{\Phi} \gamma_2(t) :=
\Phi(\gamma_1(t),\gamma_2(t))
\]
and a product of curves
\[
\gamma_{1}(t) \otimes_{\Psi} \gamma_2(t) :=
\Psi(\gamma_1(t),\gamma_2(t))\,.
\]
Therefore, the set of all power series (without constant term) turns out to be a ring, that we
shall denote by $\mathscr{C}(\Phi,\Psi; A)$. This ring is commutative
assuming that both $\Phi$ and $\Psi$ are.  The zero element of
$\mathscr{C}(\Phi,\Psi; A)$ is represented by the zero power
series. Also, as is well known, there exists a power series $i(x)$
such that $\Phi(x,i(x))=0$. This allows us to define the inverse as
\[
\gamma(t) +_{\Phi} i(\gamma(t))=0\,.
\]
Notice that, although the ring of curves $\mathscr{C}(\Phi,\Psi; A)$ can not be unital,
since all formal series considered are of the form $f(x)= x +O(x^2)$
the expression $\gamma_{1}(t) \otimes_{\Psi} 1$ (where $1$ is the
multiplicative unity in $A$) still makes perfect sense.

\subsection{A filtration of subrings} Let $\mathcal{R}$ be an
$n$-dimensional formal ring over $A$. Given a curve in $\mathcal{R}$
of the form $\gamma(t)=\sum_{n=1}^{\infty} c_{n} t^{n}$,
$c_{n}\in A^{n}$, we shall consider the subset of curves whose first
$n-1$ coefficient vanish: $c_i=0$, $i=1,\ldots,n-1$. Let us denote by
$\mathscr{C}^{n}(\Phi,\Psi; A)$ the subset of such curves.  It is easy
to prove that $\mathscr{C}^{n}(\Phi,\Psi; A)$ is a subring of
$\mathscr{C}(\Phi,\Psi; A)$. Therefore, we can define the following
filtration of $\mathscr{C}(\Phi,\Psi; A)$ by subrings
\begin{equation}\label{filtr}
\mathscr{C}(\Phi,\Psi; A)= \mathscr{C}^{1}(\Phi,
\Psi; A) \supset \mathscr{C}^{2}(\Phi, \Psi; A)\supset \cdots
\supset\mathscr{C}^{n}(\Phi,\Psi; A)\supset \cdots
\end{equation}
with the property
\[
\bigcap_{n} \mathscr{C}^{n}(\Phi,\Psi; A) = 0.
\]
This filtration endows naturally the set $\mathscr{C}(\Phi,\Psi; A)$
with a topology.  It is interesting to observe that given a formal
group law $\Phi(x,y)$, the construction~\eqref{filtr} is not unique:
for any $a\in A$ there exists a filtration of subrings
$\mathscr{C}(\Phi,\Psi_{a}; A)$, that we shall call an $a$-filtration.

\section{Witt vectors, the Witt ring and formal ring structures} \label{sec.Witt}
\subsection{Generalized Witt's equations}
There is an interesting relation between the notions of formal ring and the celebrated  ring of Witt vectors, introduced in~\cite{Witt1936} in relation with unramified extensions of $p$-adic number fields. Witt vectors have also found applications in the study of algebraic varieties over a field of positive characteristic~\cite{Mumford1966} and in the theory of commutative algebraic groups~\cite{Serre1959}. Also, they have been related to the theory of formal groups~\cite{Haze}.
Precisely, the aim of this section is  to prove that the ring of Witt vectors can be regarded as a particular instance of a general class of formal ring structures.

We recall that given a an associative, commutative ring $A$ with unit element, Witt vectors are infinite sequences $a=(a_i)_{i\in\NN}\subset A$, added and multiplied according to the Witt rules (for details, see~\cite{Lang1974}).

First of all we shall derive Witt's construction from the point of view of formal rings. To this purpose, let us start illustrating the {\em truncated} construction for fixed $n\in\NN\setminus\{0\}$. Given a group logarithm $G=G(\bx)$, the formal group law associated satisfies $G\circ\Phi(\bx,\by)=G(\bx)+G(\by)$, i.e., for all $1\le i\le n$ we have
\begin{equation}\label{eq:witt1}
G_i(\Phi_1(\bx,\by),\ldots,\Phi_n(\bx,\by))=G_i(\bx)+G_i(\by) \ .
\end{equation}
We shall assume (this is indeed the case in the Witt construction) that the group logarithm has the form
\begin{equation}\label{eq:gtriang}
\begin{aligned}
G_1(\bx)&=g_1(x_1) \\
G_2(\bx)&=g_2(x_1,x_2) \\
&\,\,\vdots\\
G_n(\bx)&=g_n(x_1,\ldots,x_n)
\end{aligned}
\end{equation}
where $(g_1,\ldots,g_n)$ are formal series, each having the form~\eqref{Gindim}. Thus, from Eq.~\eqref{eq:witt1} we deduce $g_1(\Phi_1(\bx,\by))=g_1(x_1)+g_1(y_1)$
which implies that $\Phi_1$ exists and it only depends on $(x_1, y_1)$.
Iterating this argument, we observe that
\begin{equation}\label{eq:fWE}
g_i(\Phi_1,\ldots,\Phi_i)=g_i(x_1,\ldots,x_i)+g_i(y_1,\ldots,y_i)
\end{equation}
which implies that, if it exists, $\Phi_i(\bx,\by)\equiv\Phi_i(x_1,\ldots,x_i,y_1,\ldots,y_i)$. Equations~\eqref{eq:fWE} for all $i$ will be called the \textit{first generalized Witt equations}: they generalize the original construction due to Witt for the first composition law, and determine a typical {\em triangular} structure in the arguments of $\Phi$.
\par
Concerning the second composition law, by means of the same reasoning we obtain the \textit{second generalized Witt equations}
\begin{equation}\label{eq:sWE}
g_i(\Psi_1,\ldots,\Psi_n)=g_i(x_1,\ldots,x_i)\cdot g_i(y_1,\ldots,y_n)
\end{equation}
which imply $\Psi_i(\bx,\by)\equiv\Psi_i(x_1,\ldots,x_i,y_1,\ldots,y_i)$.


\subsection{Two families of solutions to the generalized Witt's equations}
The generalized Witt equations~\eqref{eq:fWE}-\eqref{eq:sWE} admit at least a family of separated solutions that can be obtained in an iterative way, as we will show. Consider for all $n\in\mathbb{N}\setminus\{0\}$ the {\em generalized  ghosts} defined by
\begin{equation}\label{eq:chif}
g_n(x_1,\ldots,x_n):=\sum_{i=1}^{n}\chi^{(n)}_i(x_i),\qquad\chi^{(n)}_i(x)\in A\llbracket x\rrbracket\,.
\end{equation}
Note that Eqs.~\eqref{Gindim}-\eqref{eq:gtriang} imply $\chi^{(n)}_n(x)=x+O(x^2)$.
Explicitly,
$\Phi_1(x_1,y_1)=({\chi_1^{(1)}})^{-1}(\chi_1^{(1)}(x_1)+\chi_1^{(1)}(y_1))$ and $\Psi_1(x_1,y_1)=({\chi_1^{(1)}})^{-1}(\chi_1^{(1)}(x_1)\cdot \chi_1^{(1)}(y_1))$. These are the initial conditions of the recursion relations
\begin{align}
\chi_n^{(n)}\circ\Phi_n&=\sum_{i=1}^{n}\left(\chi^{(n)}_i(x_i)+\chi^{(n)}_i(y_i)\right)-\sum_{i=1}^{n-1}\chi^{(n)}_i\circ\Phi_i\label{eq:rechif1}\\
\chi_n^{(n)}\circ\Psi_n&=\sum_{i,j=1}^{n}\chi^{(n)}_i(x_i)\cdot\chi^{(n)}_j(y_j)-\sum_{i=1}^{n-1}\chi^{(n)}_i\circ\Psi_i\label{eq:rechif2}
\end{align}
which, since $\chi^{(n)}_n$ is invertible, provide the explicit solutions $\Phi_n$ and $\Psi_n$ of Eqs.~ \eqref{eq:fWE}-\eqref{eq:sWE}. Thus, the series $g_n(x_1,\ldots,x_n)\in A\llbracket x_1,\ldots,x_n \rrbracket$, once substituted into Eqs.~\eqref{eq:gtriang} define a formal ring structure {\em for all} $n\in\NN\setminus\{0\}$, according to Theorem \ref{theomain}. The choice~\eqref{eq:chif}, interesting for its simplicity, is obviously just one of the many possibilities allowed. A much more general solution is the partially separated one, corresponding to
\begin{equation}\label{eq:hchi}
\begin{aligned}
g_{n}(x_1,\ldots,x_n)&:= \tau_{n-1}(x_1,\ldots x_{n-1})+ \omega_{n}(x_n)\ , \qquad n>1\\
g_1(x_1)&:=\omega_1(x_1)\ .
\end{aligned}
\end{equation}
As above, we have $\omega_n(x)=x+O(x^2)$ for all $n$. Again, $\Phi_1$ and $\Psi_1$ are explicitly determined as the initial conditions of the recursion relations
\begin{align}
\omega_n\circ\Phi_n&=\tau_{n-1}(x_1,\ldots,x_{n-1})+\tau_{n-1}(y_1,\ldots,y_{n-1})\label{eq:rehchi1}\\
&\hphantom{=\tau_{n-1}(x_1,\ldots,x_{n-1})}+\omega_n(x_n)+\omega_n(y_n)-\tau_{n-1}(\Phi_1,\ldots,\Phi_{n-1})\,,\nn\\[5pt]
\omega_n\circ\Psi_n&=\tau_{n-1}(x_1,\ldots,x_{n-1})\cdot\tau_{n-1}(y_1,\ldots,y_{n-1})\label{eq:rehchi2}\\
&\hphantom{=\tau}+\tau_{n-1}(x_1,\ldots,x_n)\cdot\omega_n(y_n)+\tau_{n-1}(y_1,\ldots,y_n)\cdot\omega_n(x_n)\nn\\
&\hphantom{=\tau_{n-1}(x_1,\ldots,x_n)\cdot\tau_{n-1}}+\omega_n(x_n)\cdot\omega_n(y_n)-\tau_{n-1}(\Psi_1,\ldots,\Psi_{n-1})\nn\,.
\end{align}

For any natural number $n$, considerations above imply that any $n$-dimensional formal ring defines a new ring structure over the set of truncated Witt vectors of length $n$, namely finite $n$-tuples $a=(a_1,\ldots,a_n)$, $a_i\in A$. 

\begin{definition}\label{def:GWnA}
Let $\mathcal{R}=(A,\Phi,\Psi)$ be an $n$-dimensional formal ring over $A$ such that its group logarithm is of the form~\eqref{eq:gtriang}. The corresponding generalized truncated Witt ring $GW_n(A)$ of truncated Witt vectors of length $n$ over $A$ is the ring defined by
\begin{equation}\label{eq:GWnA}
\begin{aligned}
(a_1,\ldots,a_n)\oplus(b_1,\ldots,b_n)&=\left(\mss\Phi_1(a_1,b_1),\ldots,\Phi_n(a_1,\ldots,a_n,b_1,\ldots,b_n)\mss\right),\\
(a_1,\ldots,a_n)\otimes(b_1,\ldots,b_n)&=\left(\mss\Psi_1(a_1,b_1),\ldots,\Psi_n(a_1,\ldots,a_n,b_1,\ldots,b_n)\mss\right).\\
\end{aligned}
\end{equation}
The components~\eqref{eq:gtriang} of the group logarithm underlying the ring $\mathcal{R}$ will be called the generalized ghosts of the truncated Witt vectors.
\end{definition}

It is straightforward to proof that $GW_n(A)$ is well-defined as a ring. Associativity and distributivity in $GW_n(A)$ are a direct consequence of the analogous properties valid in the formal ring $\mathcal{R}$.

Motivated by the previous discussion, one can introduce a general ring structure over the set of infinite Witt sequences $a=(a_i)_{i\in\NN}\subset A$. 

\begin{definition}\label{def:GWA}
The generalized Witt ring $GW(A)$ of Witt vectors over $A$ is the ring defined by
\begin{equation}\label{eq:GWA}
\begin{aligned}
(a_1,a_2,\ldots)\oplus(b_1,b_2,\ldots)&=\left(\mss\Phi_1(a_1,b_1),\Phi_2(a_1,a_2,b_1,b_2),\ldots\mss\right),\\
(a_1,a_2,\ldots)\otimes(b_1,b_2,\ldots)&=\left(\mss\Psi_1(a_1,b_1),\Psi_2(a_1,a_2,b_1,b_2),\ldots\mss\right).\\
\end{aligned}
\end{equation}
\end{definition}

\begin{remark}
In $GW(A)$ the zero element and the unit are given, respectively, by $(0,0,\ldots)$ and $(1,0,\ldots)$ whereas their truncated versions of length $n$ give the corresponding elements in $GW_n(A)$.
Finally, note that the previous definitions imply the existence of formal power series $r_n(x_1,\ldots,x_n)\in A\llbracket x_1,\ldots,x_n\rrbracket$ for all $n$ such that
\[
(a_1,a_2,\ldots)\oplus(r_1(a_1),r_2(a_1,a_2),\ldots)=(0,0,\ldots) \ .
\]
\end{remark}

\subsection{The Witt ring W(A) as a formal ring}
The previous approach allows us to incorporate   the standard Witt construction into the general framework provided by formal ring theory. Indeed,  the well-known ring of Witt vectors $W(A)$ is recovered straightforwardly, if $p$ is a prime number, by means of the choice
\beq \label{Wittpol}
g_n(x_1,\ldots,x_n):=x_1^{p^{n-1}}+px_2^{p^{n-2}}+\cdots+p^{n-1}x_n
\eeq
in Eqs.~\eqref{eq:gtriang}, or equivalently
$
\chi_{i}^{(n)}(x) = p^{i-1} x^{p^{n-i}}
$
in Eqs.~\eqref{eq:chif}. In other words, if the components of the group logarithm~\eqref{eq:gtriang} coincide with the usual ``ghosts'' (or Witt polynomials) associated with Witt vectors~\cite{Lang1974}, then we obtain the standard ring $W(A)$, or by truncation the corresponding ring $W_n(A)$. In order to illustrate this point, let us compute explicitly this ring structure for $n=2$. We choose
\begin{align*}
G_1(x_1,x_2)&=g_1(x_1)=x_1\\
G_2(x_1,x_2)&=g_2(x_1,x_2)=x_1^p+px_2
\end{align*}
according to the triangular structure~\eqref{eq:gtriang}. Therefore we have the following expressions for the components of the compositional inverse of $G=(G_1,G_2)$,
\begin{align*}
(G^{-1})_1(x_1,x_2)&=x_1\\
(G^{-1})_2(x_1,x_2)&=-\frac1px_1^p+\frac1px_2.
\end{align*}
A direct substitution provides the following solution to the first generalized Witt's equations~\eqref{eq:fWE}:
\begin{align*}
\Phi_1(x_1,y_1)&=x_1+y_1,\\
\Phi_2(x_1,x_2,y_1,y_2)&=x_2+y_2-\sum_{i=1}^{p-1}\frac{1}{p}\binom{p}{k}x_{1}^{k}\cdot y_{1}^{p-k}\,.
\end{align*}
Exactly in the same manner, the solution to the second generalized Witt's equations~\eqref{eq:sWE} is provided by:
\begin{align*}
\Psi_1(x_1,y_1)&=x_1\cdot y_1,\\
\Psi_2(x_1,x_2,y_1,y_2)&=x_{1}^{p}\cdot y_2+x_2\cdot y_1^p+p\,x_2\cdot y_2\,.
\end{align*}
\subsection{Universal Witt vectors}
As is well known, the Witt polynomials \eqref{Wittpol} for a fixed prime $p$ are special cases of the construction of the \textit{universal Witt polynomials}, valid for any non-zero natural number. Consequently, a universal Witt ring $UW(A)$ of Witt vectors over $A$ (not depending on a choice of a prime $p$) can be constructed \cite{Lang1974}.
We observe that the universal ring $UW(A)$  can also be recovered as a particular realization of the ring $GW(A)$ by choosing for the components of the group logarithm~\eqref{eq:gtriang} the ghosts
\[
g_n(x_1,\ldots,x_n)=\sum_{d|n}dx_d^{n/d} \ .
\]Note that this choice is in agreement with Eq.~\eqref{Gindim}.

\section{Hirzebruch's multiplicative classes and formal rings.}\label{sec.hir}
In this section we shall establish a connection between the notion of
formal rings and the well-known theory of genera of Hirzebruch, which
is of special interest in algebraic topology~\cite{Hirzebruch1978}.
Precisely, we will construct the formal ring structures associated
with some families of formal groups that are related with
multiplicative genera of relevance in cobordism theory.

\subsection{The Todd genus}
Let $\Omega_{U}$ denote the ring of geometric cobordisms \cite{Quillen}. Then the  genus $T: \Omega_{U} \to \mathbb{Z}$ is the genus of the characteristic formal power series
\[
T(z)=  \frac{-z}{1-e^{-z}}=\frac{-z}{G^{-1}_{T}(z)} \ ,
\]
with the corresponding formal group law given by
$\Phi_{T}(x,y)=G_{T}^{-1}\big(G_{T}(x)+ G_{T}(y)\big)$, with $G_{T}(x)=-\ln(1-x)$, so that
$\Phi_{T}(x,y)= x+y-xy$. Slightly more generally, we shall consider
the one-parametric multiplicative group law
\begin{equation}\label{1dimmultlaw}
\Phi(x,y)= x+y +\alpha xy.
\end{equation}
Recently, this group law has been intimately related to statistical
mechanics, since it represents the composition law associated with the
Tsallis entropy \cite{P2016}. Here we shall assume
$\alpha\in\mathbb{R}$, $\alpha\neq 0$.  Now we have
\[
  G(s)=(1/\alpha)\log(1+\alpha s), \qquad G^{-1}(t)=(e^{\alpha
    t}-1)/\alpha.
\]
Consequently, the formal ring structure associated is provided by the
multiplicative law \eqref{1dimmultlaw} jointly with the product
\[
  \Psi(x,y)=\frac{\exp\big((1/\alpha)\log(1+\alpha x)\log(1+\alpha
    y)\big)-1}\alpha \ ,\qquad
\]
where formally
\[
  \log(1+X): = \sum_{n=1}^{\infty}(-1)^{n+1}\frac{X^{n}}{n}.
\]

\subsection{The $T_{q}$ genus} The law of multiplication associated
\[
  \Phi_{T_{q}}(x,y)= \frac{x+y+(q-1)xy}{1+qxy}
\]
is defined over the ring $\mathbb{Z}\llbracket q\rrbracket$. For
$q=-1,0,1$ we define the genera $c$, $T$ and $L$ respectively.  For
the $c$ genus, $c: \Omega_{U}\to\mathbb{Z}$, we have
\[
  \Phi_{c}(x,y)= \frac{x+y-2xy}{1-xy}, \quad G_{c}(s)= \frac{s}{1-s},
  \quad G_{c}^{-1}(t)= \frac{t}{1+t}\ .
\]
Therefore
\[
  \Psi_{c}(x,y)= \frac{xy}{1-x-y+2xy} \ .
\]
For $q=0$, we have the case of Eq.~\eqref{1dimmultlaw} with $\alpha=-1$ discussed above. For the $L$ genus we obtain
\[
  \Phi_{L}(x,y)= \frac{x+y}{1+xy}, \quad
  G_{L}(s)=\tanh^{-1}(s)=\frac12\log\bigg(\frac{1+s}{1-s}\bigg) \ .
\]
Thus we deduce the composition law
\[
  \Psi_{L}(x,y)=\tanh\left(\frac14\log\left(\frac{1+x}{1-x}\right)\log\left(\frac{1+y}{1-y}\right)\right)\,
  .  
\]

\subsection{The Euler and Abel formal rings} The Euler group law is
defined by
\[
  \Phi_E(x,y)=\frac{x\sqrt{1-y^4}+y\sqrt{1-x^4}}{1+x^2 y^2}\,.
\]
Let us introduce $f(s)=(1-s^4)^{-1/2}$ and $G_{E}(x)=\int_0^x
f(s)\,\mathrm{d}s$. We obtain
\[
  \Phi_E(x,y)=G_E^{-1}\left(G_E(x)+G_E(y)\right)\,.
\]
To determine $G_E^{-1}(x)$, it is sufficient to express $G_E(x)$ as a
formal power series and then to compute its compositional
inverse. Thus, expanding the integrand $f(s)=(1-s^4)^{-1/2}$ and
performing a formal integration term by term, we obtain the formal
series
\begin{align*}
  G_E(x)&=\int_0^x\sum_{k=0}^\infty\frac{(2k-1)!!}{2k!!}s^{4k}\,\mathrm{d}s=\sum_{k=0}^\infty\frac1{4k+1}\frac{(2k-1)!!}{(2k)!!}x^{4k+1}\\
        &=x+\frac{x^5}{10}+\frac{x^9}{24}+\frac{5\,x^{13}}{208}+\frac{35\,x^{17}}{2176}+\cdots.
\end{align*}
We deduce
\[
  G_{E}^{-1}(x)=x-\frac{x^5}{10}+\frac{x^9}{120}-\frac{11\,x^{13}}{15600}+\frac{211\,x^{17}}{3536000}+\cdots.
\]
The product law is immediately constructed by means of
Eq.~\eqref{Psi}. The first few terms read
\begin{multline*}
  \Psi_{E}(x,y)=xy+\frac{xy^5+x^5y}{10}+\frac{xy^9+x^9y}{24}-\frac{9\,x^5y^5}{100}\\
  -\frac{11\left(x^5y^9+x^9y^5\right)}{240}+\frac{5\left(xy^{13}+x^{13}y\right)}{208}+\cdots.
\end{multline*}
Note that $G_{\it{E}}^{-1}(1)$ satisfies $0< G_{\it{E}}^{-1}(1) <1$
since it is the solution of the equation $g(x)=0$ with
$g(x)=G_{\it{E}}(x)-1$.  This is because
$G_{\it{E}}(x)=\int_0^x f(s)\,\mathrm{d}s$ is continuous and monotonically
increasing and indeed $G_{\it{E}}(0)=0$ and $G_{\it{E}}(1)>1$ so we
have $g(0)=-1$ and $g(1)>0$.

\subsubsection{The Abel formal ring} The Abel formal group law has
been intensively investigated due to its prominence in algebraic
topology~\cite{Busato} and~\cite{Haze}. It is defined by
\[
  \Phi_{A}(x,y)=G_{A}^{-1}\left(G_{A}(x)+G_{A}(y)\right),\qquad
  G_{A}^{-1}(t)=\frac{e^{at}-e^{bt}}{a-b}, \qquad a\neq b\,,
\]
with
\[
  G_A(x)=x-\frac12(a+b)\,x^2+\frac16(2\,a^2+5\,ab+2\,b^2)\,x^3+\cdots.
\]
If $a=b$ the group exponential coincides with the classical Abel's
function
\[
  G_{A}^{-1}(t)= te^{at}.
\]
From Theorem~\ref{teo} we deduce the following form of the composition
law:
\begin{multline*}
  \Psi_{A}(x,y)=xy-\frac12(a+b)(xy^2+x^2y)+\frac16(2a^2+5ab+2b^2)(xy^3+x^3y)+\\
  +\frac14\left(a^2+2a(b+1)+b(b+2)\right)x^2y^2+\cdots \ .
\end{multline*}

\section{The formal ring of a two-dimensional multiplicative group law}\label{sec.2dim}
We shall present here an example of a two-dimensional formal
ring. First we construct a two-dimensional, bi-parametric version of
the multiplicative formal group law.  One can directly prove that the
series
\begin{align}
  \Phi_1(x_1,x_2,y_1,y_2)&:=x_1+y_1+ a x_1y_1+ b\big(x_1 y_2+x_2 y_1\big)+ a x_2y_2 \label{2dimfg1}\\
  \Phi_2(x_1,x_2,y_1,y_2)&:= x_2+y_2+ b x_1 y_1+ a \big(x_1 y_2+x_2 y_1 \big)+ b x_2y_2 \label{2dimfg2}
\end{align}
with $a$, $b \in \mathbb{R}$, define a formal group law
$\Phi(\bx,\by)=(\Phi_1(\bx,\by),\Phi_2(\bx,\by))$ which, to our knowledge, is new. The associated 2-dimensional group exponential
$G^{-1}(t_1,t_2)$ is given
explicitly by:
\[
  \begin{aligned}
    (G^{-1})_{1}(t_1,t_2)&=\frac{e^{2(a+b)(t_1+t_2)}-1}{2(a+b)}+\frac{e^{2(a-b)(t_1-t_2)}-1}{2(a-b)}\,,\\
    (G^{-1})_{2}(t_1,t_2)&=\frac{e^{2(a+b)(t_1+t_2)}-1}{2(a+b)}-\frac{e^{2(a-b)(t_1-t_2)}-1}{2(a-b)}\
    ,
  \end{aligned}
\]
with $G^{-1}(G(\bx)+G(\by))=(\Phi_1(\bx,\by),\Phi_2(\bx,\by))$. According to Theorem~\ref{theomain}, the existence of a ring structure corresponding to the two-dimensional
formal group~\eqref{2dimfg1}-\eqref{2dimfg2} is ensured by the series
$\Psi(\bx,\by)=(\Psi_1(\bx,\by),\Psi_2(\bx,\by))$,
where
\[
  \begin{aligned}
    \Psi_1(x_1,x_2,y_1,y_2)&=\psi(x_1,x_2,y_1,y_2)+\varphi(x_1,x_2,y_1,y_2)\\
    \Psi_2(x_1,x_2,y_1,y_2)&=\psi(x_1,x_2,y_1,y_2)-\varphi(x_1,x_2,y_1,y_2)
  \end{aligned}
\]
and finally,
\[
  \begin{aligned}
    \psi(x_1,x_2,y_1,y_2)&=\frac{1}{a}\exp\left(\frac{a\ln\big(1+\frac{1}{2}b(x_1-x_2)\big)\ln\big(1+\frac{1}{2}b(y_1- y_2)\big)}{2b^2}\right)\times\\
    &\times\exp\left(\frac{\ln\big(1+\frac{1}{2}a(x_1+ x_2)\big)+\ln\big(1+\frac{1}{2}a(y_1+y_2)\big)}{2a}\right)-\frac{1}{a}\\
    \varphi(x_1,x_2,y_1,y_2)&=\frac{1}{b}\exp\left(\frac{\ln\big(1+\frac{1}{2}a(x_1+x_2)\big)\ln\big(1+\frac{1}{2}b(y_1-y_2)\big)}{2a}\right)\times\\
    &\times\exp\left(\frac{\ln\big(1+\frac{1}{2}b(x_1-
        x_2)\big)+\ln\big(1+\frac{1}{2}a(y_1+y_2)\big)}{2a}\right)-\frac{1}{b}\,.
  \end{aligned}
\]


\section*{Acknowledgement}
This work has been partly supported by the research project
FIS2015-63966, MINECO, Spain, and by the ICMAT Severo Ochoa project
SEV-2015-0554 (MINECO). P.T. is member of the Gruppo Nazionale di Fisica Matematica (INDAM), Italy.

\end{document}